\documentclass[12pt,dvips]{amsart}

\usepackage{amsmath}
\usepackage{amssymb}
\usepackage{amscd}
\usepackage{latexsym}
\usepackage{palatino}
\usepackage{amsthm} 
\usepackage{epsfig}
\usepackage{psfrag}
\usepackage{amsfonts}
\usepackage{enumerate}
\usepackage{bbold}
\usepackage{comment}
\usepackage{nomencl}

\numberwithin{equation}{section}
\numberwithin{figure}{section}

\theoremstyle{plain}

\newtheorem{theorem}{Theorem}[section]
\newtheorem{lemma}[theorem]{Lemma}

\newtheorem{proposition}[theorem]{Proposition}

\newtheorem{corollary}[theorem]{Corollary}

\theoremstyle{definition}

\newtheorem{definition}[theorem]{Definition}

\theoremstyle{remark}

\newtheorem{remark}[theorem]{Remark}

\input xy
 \xyoption{all}

\newcommand{\id}{\operatorname{id}}

\newcommand{\End}{\operatorname{End}}
\newcommand{\ad}{\operatorname{ad}}

\DeclareMathOperator{\tr}{tr}

\newcommand{\R}{{\mathbb R}}
\newcommand{\C}{{\mathbb C}}

\UseComputerModernTips

\DeclareMathOperator{\GL}{GL}
\DeclareMathOperator{\SL}{SL}
\DeclareMathOperator{\PSL}{PSL}
\DeclareMathOperator{\U}{U}
\DeclareMathOperator{\SU}{SU}

\begin{document}

\title[Higgs bundles over cell complexes]{Higgs bundles over cell complexes and representations of finitely presented groups}

\author{Georgios  Daskalopoulos}
\address{Department of Mathematics, Brown University, Providence, RI 02912, USA}
\email{daskal@math.brown.edu}

\author{Chikako Mese}
\address{Department of Mathematics, Johns Hopkins University, Baltimore, MD 21218, USA}
\email{cmese@math.jhu.edu}
  
  \author{Graeme Wilkin}
\address{Department of Mathematics, National University of Singapore, Singapore 119076}
\email{graeme@nus.edu.sg}



\begin{abstract}
The purpose of this paper is to extend the Donaldson-Corlette theorem to the case of vector bundles over cell complexes. We define the notion of a vector bundle and a Higgs bundle over a complex, and describe the associated Betti, de Rham and Higgs moduli spaces. The main theorem is that the $\SL(r, \C)$ character variety of a finitely presented group $\Gamma$ is homeomorphic to the moduli space of rank $r$ Higgs bundles over an admissible   complex $X$ with $\pi_1(X) = \Gamma$. A key role is played by  the theory of harmonic maps defined on singular domains.
\end{abstract}

\maketitle


\section{Introduction}\label{sec:intro}

Higgs bundles  were first introduced by Hitchin in \cite{hitchin} as a PDE  on a vector bundle over a Riemann surface obtained by the dimensional reduction of the anti-self dual equations on $\R^4$. Since then, the field has seen a remarkable explosion in different directions most notable the work of Simpson \cite{Simpson88} on variations of Hodge structures. The work of Donaldson \cite{Donaldson87} and Corlette \cite{Corlette88} provided links with the theory of flat bundles and  character varieties of groups. Higgs bundles have been generalized over noncompact manifolds \cite{Simpson90}, \cite{JostZuo96}, \cite{JostYangZuo07}, \cite{CorletteSimpson08} and singular projective varieties  \cite{pantev}. The goal of this paper is to push this even further by considering Higgs bundles over more general singular spaces; namely,  finite simplicial complexes. 

As pointed out  by Hitchin, Donaldson and Corlette, a key role in the relation between character varieties and Higgs bundles is played by  the theory of harmonic maps. Harmonic maps had been used before in the study of representations of 
K\"ahler manifold groups starting with  the  work of Siu \cite{siu} and has seen some remarkable applications in providing new proofs of the celebrated Margulis super rigidity theorem (cf. \cite{jost} and all the references therein) and the only existing proof of  the rank one superrigidity theorem due to Corlette and Gromov-Schoen (cf. \cite{Corlette92} and \cite{gromov-schoen}). But these directions involved showing that the representations are rigid, in contrast  with Hitchin's point of view which is to study the moduli space of such representations. 

In all the above references, one studies representations of fundamental groups of smooth manifolds  rather than arbitrary finitely presented groups. In other words, the domain space of the harmonic map is smooth. Chen and  Eells-Fuglede  (cf. \cite{chen}, \cite{eells-fuglede}) developed the theory of harmonic maps from a certain class of singular domains including admissible simplicial complexes. By admissible they mean complexes that are dimensionally homogeneous and locally chainable in order to avoid certain analytic pathologies (see next section for precise definitions). Since any finitely presented group is the fundamental group of an admissible complex, there is no real restriction in considering admissible complexes. The key property of harmonic maps shown in the above references is that they are H\"older continuous but in general they fail to be Lipschitz. In fact, the work of the first two authors shows that Lipschitz harmonic maps often imply that the representations are rigid (cf. \cite{DaskalMese08} and \cite{DaskalMese09}).

The starting point of this paper is a finitely presented group $\Gamma$ and a two-dimensional admissible complex  without boundary $X$ with fundamental group $\pi_1(X)\simeq \Gamma$. We also fix  a piecewise smooth vector bundle $E$ over 
$X$ that admits a flat $\SL(r, \C)$-structure. Such bundles are parametrized topologically by the (finitely many) connected components of the $\SL(r, \C)$-character variety of $\pi_1(X)$. One can write down Hitchin's equations 
\begin{align} 
F_A + \psi \wedge \psi & = 0 \label{eqn:mu1'-intro}  \\
d_A \psi & = 0 \label{eqn:mu2'-intro}
\end{align}
for a sufficiently regular unitary connection  $A$ and Higgs field $\psi$. Again, as in the smooth case, the $\SL(r, \C)$ connection $d_A + \psi $ is flat and one can ask  what the precise condition is so that the pair $(d_A , \psi) $ corresponds to a representation $\rho: \pi_1(X)\rightarrow \SL(r, \C)$.

Given a representation $\rho$ as above, we can associate as in the smooth case a $\rho$-equivariant harmonic map from the universal cover $\tilde X$ to the symmetric space $\SL(r, \C) / \SU(r)$. The first two authors in  \cite{DaskalMese08}  studied harmonic maps from simplicial complexes to smooth manifolds and discovered the following crucial properties:
\begin{enumerate}

\item The harmonic map is smooth away from the codimension 2 skeleton of $\tilde X$.
\item The harmonic map satisfies a balancing condition at the codimension 1 skeleton of $\tilde X$ in the sense that the sum of the normal derivatives vanishes identically.
\item The harmonic map blows up in a controlled way at the codimension 2 skeleton of $\tilde X$.

\end{enumerate}

All the above properties are described precisely in Theorem~\ref{thm:harmonic-map-existence}.  This allows us to prove that  the derivative of the harmonic map belongs in  an appropriate  weighted Sobolev space $L_{1, \delta}^2$ (cf. Proposition~\ref{prop:rightspace}). The definition of weighted  Sobolev spaces is given in Section \ref{subsec:weightedsobolev}.  Finally, the main theorem describing the correspondence between equivalence classes of balanced Higgs pairs of class $L_{1, \delta}^2$ and representations is given  in Section \ref{sec:equivalence} (cf. Theorem~\ref{thm:BettiHiggshomeomorhism}).

\section{Vector bundles over complexes}

\subsection{Basic Definitions of  smooth bundles}\label{subsec:basic-definitions}

\begin{definition}  Let ${\bf E}^N$ be an $N$-dimensional affine space.  A convex  linear polyhedron $\sigma$ of dimension $i$ with interior in some ${\bf E}^i \subset  {\bf E}^N$ is called a cell of dimension $i$.  We will use the notation $\sigma^i$ to denote a cell $\sigma$ of dimension $i$.  A \emph{convex cell complex} or simply a \emph{complex} $X$ in  ${\bf E}^N$ is a finite collection  ${\mathcal F}= \{ \sigma \}$ of  cells  satisfying the following properties:  $(i)$ the boundary $\partial{\sigma^i}$ of $\sigma^i \in {\mathcal F}$ is a union 
of $\tau^j \in {\mathcal F}$ with $j<i$ (called the faces of $\sigma^i$) and $(ii)$ if  $\tau^j, \sigma^i \in {\mathcal F}$ with $j<i$ and ${\sigma^i\cap{\tau^j}}\neq\emptyset$, then  $\tau^j
\subset{\sigma^i}$. \emph{The dimension of a complex} $X$ is the maximum dimension of a cell in $X$.
For a cell $\sigma^i$, we call ${\bf E}^i$ the \emph{extended plane} defined by $\sigma_i$
\end{definition}
For example, a simplicial complex is a cell complex whose cells are all simplices.

\begin{definition}
A connected complex $X$ of dimension $n$ is said to be admissible (cf. \cite{chen} and \cite{eells-fuglede}) if the
following two conditions hold:\\
\\
(i)  $X$ is dimensionally homogeneous, i.e., every cell is
contained in a closure of at least one $n$-cell, and \\
\\
 (ii) $X$ is locally $(n-1)$-chainable, i.e., given any $(n-2)$-cell $v$, every two $n$-cells $\sigma$
and $\sigma'$ incident to  $v$ can be joined by a sequence
$\sigma=\sigma_0,\sigma_1,..., \sigma_k=\sigma'$ where $\sigma_i$ and $\sigma_{i+1}$ are two adjacent $n$-cells incident to $v$ for $i=0,1, \dots, k-1$.
\end{definition}

The boundary $\partial X$ of $X$ is the union of  the closures of the $(n-1)$-cells contained in the closure of exactly one $n$-cell.    
 We will  assume that $X$ is  equipped  with a Riemannian metric $g_{\sigma}$ on the closure $\overline \sigma$ of each $n$-cell $\sigma$ of $X$ satisfying the additional property that  if $\tau$  is a face of $\sigma$, then $g_{\sigma} \big|_{\tau}=g_{\tau}$.   We endow $\GL(n,\C)/\U(n)$ with a Riemannian metric of non-positive sectional curvature such that $\GL(n,\C)$ acts by isometries.

We note that starting in Section~\ref{hmhb}, we will restrict to the case when $X$ is a finite  admissible simplicial complex of dimension 2 without boundary.

 \begin{definition}
Let $U$ be an open subset of $X$.  A function $f: U\rightarrow \R$ is \emph{ smooth}  if for any $n$-cell $\sigma$,   $f$ is smooth on  $\overline{\sigma} \cap U$. More precisely, we say that $f$ is smooth on $U$ if for any cell  $\sigma$, the restriction $f\big|_{\overline{\sigma} \cap U}$ can be extended  to a smooth function in a neighborhood of $\overline{\sigma} \cap U$ in the extended plane defined by $\sigma$.  
 Let $f:U \rightarrow Y \subset {\bf E}^M$ be a map into a  complex $Y$ and write $f=(f^1,\dots,f^M)$ in terms of an affine coordinate system  on ${\bf E}^M$.  The map $f$ is \emph{ smooth} if $f^j$ is  smooth for every $j=1, \dots, M$.  
\end{definition}

Now we can use these definitions to define the notion of a  smooth vector bundle over a complex.

\begin{definition}
A surjective map $\pi : E \rightarrow X$ is a \emph{ smooth complex vector bundle of rank $r$} over a  complex $X$ if there exists an open cover $\{ U_\alpha \}_{\alpha \in I}$ of $X$ such that
\begin{enumerate}

\item For each $\alpha \in I$, there exists a  homeomorphism $\varphi_\alpha : \pi^{-1}(U_\alpha) \rightarrow U_\alpha \times \C^r$.
\item  For each $x \in X$, $\varphi_\alpha$ defines a linear isomorphism $\pi^{-1}(x) \cong \C^r$ . 

\item If $U_\alpha \cap U_\beta \neq \emptyset$, then the transition function $g_{\alpha \beta} = \varphi_\alpha \circ \varphi_\beta^{-1} : U_\alpha \cap U_\beta \times \C^r \rightarrow U_\alpha \cap U_\beta \times \C^r$ is a  smooth homeomorphism.
\end{enumerate}
A \emph{ smooth section} of $\pi : E \rightarrow X$ is a  smooth map $s : X \rightarrow E$ satisfying $\pi \circ s = \id_X$. Let $\Omega^0(X, E)$ denote the vector space of all  
smooth sections of $\pi : E \rightarrow X$.    If $E$ is a  smooth vector bundle, then so is any associated bundle formed by taking the dual, tensor product, etc. In particular, if $E$ is  smooth then $\End(E)$ is  smooth.  
\end{definition}

\begin{definition}
For a complex $X$ of dimension $n$, the \emph{tangent bundle} $\pi:TX \rightarrow X$ is a vector bundle of rank $n$ such that for any $n$-cell $\sigma$ of $X$,  the map $\pi: \pi^{-1}(\overline{\sigma}) \rightarrow \overline{\sigma}$ is the restriction of the tangent bundle $T(\overline \sigma)$ of $\overline \sigma$ in the extended plane of $\sigma$.  
\end{definition}

\begin{definition}
A \emph{ smooth $p$-form} $\omega=\{\omega_{\sigma}\}$ on $X$ with values in a  smooth vector bundle $E$  is a collection of $p$-forms $\omega_{\sigma}$ for each cell $\sigma \in {\mathcal F}$  with values in $E$, smooth on  $\overline \sigma$  with the additional property that   if $\tau$ is a face of  $\sigma$, then $\omega_{\sigma}|_{\tau}=\omega_{\tau}$.  
 We denote $\Omega^p(X,E)$ as the space of all 
 smooth $p$-forms with values in $E$.  If $E$ is the trivial line bundle, then we write $\Omega^p(X)=\Omega^p(X,E)$ and this is the space of  smooth $p$-forms on $X$.  Given a smooth $p$-form $\omega=\{\omega_{\sigma}\} \in \Omega^p(X)$, we define $d\omega=\{d\omega_{\sigma}\}$ and note that this is a well-defined  smooth $(p+1)$-form.  Clearly, $d^2=0$ and the complex $(\Omega^*(X),d)$ denotes the  smooth deRham complex.  We denote by $H^p_{dR}(X)$ the cohomology groups associated with this complex (cf. \cite[Ch. VIII]{GriffithsMorgan81}).   
\end{definition}

\begin{definition}
A \emph{ smooth connection} on a  smooth vector bundle $\pi : E \rightarrow X$ is a  $\C$-linear map $D : \Omega^0(X, E) \rightarrow \Omega^1(X, E)$ that satisfies the Leibniz rule
$$
D(fs) = (df) s + f (Ds), \quad f \in \Omega^0(X), s \in \Omega^0(X, E).
$$
We denote the space of all  smooth connections by $\mathcal{A}^\C(E)$.
\end{definition}

The definition of $D$ can be extended to bundle-valued forms in the usual way.  More precisely, any element in $\sigma \in \Omega^p(X, E)$ can be written as a linear combination of elements of the form $\sigma =   s\omega$ with $\omega \in \Omega^p(X)$ and $s \in \Omega^0(X, E)$, and define
\begin{equation}
D \sigma =  s (d \omega)  +  (Ds)  \wedge \omega .
\end{equation}

\begin{remark}
Implicit in the definition of $\Omega^1(X, E)$ is that $1$-forms with values in $E$ must agree on the interfaces between the cells in the complex $X$. Therefore, the definition above implies that a connection must map sections that agree on the interfaces to bundle-valued $1$-forms that agree on the interfaces.
\end{remark}

As for the case of a smooth vector bundle over a smooth manifold, with respect to a trivialization $\varphi_\alpha : \pi^{-1}(U_\alpha) \rightarrow U_\alpha \times \C^r$, $D=d+A_{\alpha}$
where ${(A_\alpha})_{ij}$ is a complex-valued  smooth $1$-form. 
$A_\alpha$ is called the \emph{connection form} of $D$ with respect to the trivialization $\varphi_\alpha$. In different trivialization   $\varphi_\beta$ and with  $g_{ \alpha \beta} =  \varphi_\alpha  \circ \varphi_\beta^{-1}$ we have,

\begin{equation}\label{eqn:connection-transition}
A_\beta = g_{\alpha \beta}^{-1} d g_{\alpha \beta} + g_{\alpha \beta}^{-1} A_\alpha g_{\alpha \beta}
\end{equation}

\begin{definition}
The \emph{curvature} of a  smooth connection $D$ is the matrix valued $2$-form $F_D$ defined by
$$
D^2 s = F_D s, \quad \text{for all}\, s \in \Omega^0(X, E).
$$
Locally, we have ${(F_D)}_\alpha = d A_\alpha + A_\alpha \wedge A_\alpha$, where $A_\alpha$ is the connection form of $D$. Furthermore, 
\begin{equation}\label{eqn:curvature-transition}
{(F_D)}_\beta = g_{ \alpha \beta }^{-1} {(F_D)}_\alpha g_ {\alpha \beta},
\end{equation}
and so the curvature form $F_D$ is an element of $\Omega^2(X,\End(E))$.
\end{definition}

\begin{definition} The \emph{complex gauge group} is the group $\mathcal{G}^\C (E) $ of all  smooth automorphisms of $E$. If $D$ is a  smooth connection on $E$ and $g \in \mathcal{G}^\C (E) $, then we define $g(D)= g^{-1} \circ D \circ g$.  In local coordinates, the action of $\mathcal{G}^\C (E) $ on $\mathcal{A}^\C(E)$ is
\begin{equation}\label{eqn:group-action}
g (d + A_\alpha) = d+  g^{-1}dg   + g^{-1} A_\alpha g .
\end{equation}
\end{definition}

\begin{definition}
A \emph{ smooth Hermitian metric} $h=(h_{\sigma})$ on a rank $r$ complex vector bundle $\pi : E \rightarrow X$ is a  Hermitian metric $h_{\sigma}$ on each cell $\sigma$ that is smooth on  $\overline \sigma $ and so that if $\tau$ is a face of $ \sigma$, then $h_{\sigma}|_{\tau}=h_{\tau}$.   A Hermitian metric in a trivialization $\varphi_\alpha : \pi^{-1}(U_\alpha) \rightarrow U_\alpha \times \C^r$  is given locally by a  smooth map $\tilde{h}_\alpha$ from $U_\alpha$ into the positive definite matrices in $\GL(r, \C)$, and the induced inner product on the fibres of $E$ is
$$
<s_1(x), s_2(x)> = \overline{\varphi_\alpha(s_1(x))}^T \, \tilde{h}_\alpha(x) \, \varphi_\alpha(s_2(x)) \in \C.
$$
\end{definition}

\begin{definition}
A connection $D$ on a vector bundle $E$ with a Hermitian metric $h$ is a \emph{unitary connection}  if the following equation is satisfied.
$$
d \left< s_1, s_2 \right> = \left< D s_1, s_2 \right> + \left< s_1, D s_2 \right>,
$$
where $\left< \cdot, \cdot \right>$ is the pointwise inner product on the fibres of $E$ induced by the metric $h$. The space of smooth unitary connections on $E$ is denoted $\mathcal{A}(E,h)$.
If $D \in \mathcal{A}(E,h)$, then the curvature $F_D$ is a section of $\Omega^2(\ad(E))$. In other words, with respect to a unitary frame field the curvature satisfies $F_D^* = -F_D$.
\end{definition}

\begin{definition}
The \emph{unitary gauge group} $\mathcal{G}(E)$ is the subgroup of $\mathcal{G}^\C (E) $ that preserves the Hermitian metric $h$ on each fibre of $E$. The action on $\mathcal{G}(E)$ on $\mathcal{A}^\C(E)$ preserves the space $\mathcal{A}(E,h)$.
\end{definition}

\begin{definition}
A connection $D$ on a vector bundle $E$ is \emph{flat} if $F_D = 0$.  Given a flat connection, we can define the twisted deRham complex $(\Omega^*(X,E),D)$.  The cohomology groups will be denoted by $H^p(X,E)$.  
\end{definition}

\begin{definition}
A \emph{flat structure} on a vector bundle $\pi : E \rightarrow X$ is given by an open cover $\{ U_\alpha \}_{\alpha \in I}$ and trivialisations $\{ \varphi_\alpha \}_{\alpha \in I}$ for which the transition functions $g_{ \alpha \beta } = \varphi_\alpha \circ \varphi_\beta^{-1}$ are constant. A vector bundle with a flat structure is also called a \emph{flat bundle}.
\end{definition}

\begin{remark}
Equation \eqref{eqn:connection-transition} shows that the connection $D=d$ (with zero connection form) is globally defined on a flat bundle. Thus a flat bundle clearly admits a connection of curvature zero.
The converse is also true.
\end{remark}

\begin{theorem}\label{thm:flat-connection-flat-bundle}
Let $X$ be  $n$-complex, $U$ an open subset of $X$ and $E$ a smooth vector bundle with a smooth flat connection on $U$. Then $E$ admits a flat structure.
\end{theorem}
\begin{proof} Given a flat connection $D$ on $E$, fix an cell $\sigma$, a point $x_0 \in \overline \sigma \cap U$ and consider a contractible neighborhood $V_\sigma$ of $x_0$ in the extended plane of $\sigma$. Choose a local frame $s_\sigma^0$ of $E$ on $V_\sigma$ and let $A^\sigma$ be the corresponding connection form. We are assuming that the local frames $s_\sigma^0$ patch together to define a piecewise smooth frame $s_0$ in a neighborhood of $x_0$ in $X$. We are going to choose a different trivialisation $s_\sigma$ for which the connection can be written as $D = d$. This can de done by solving the equation
\begin{align}\label{eqn:locally-trivial}
\begin{split}
g_\sigma^{-1} A^\sigma g_\sigma + g_\sigma^{-1}dg_\sigma  & = 0 \\
\Leftrightarrow \quad dg_\sigma & = - A^\sigma g_\sigma
\end{split}
\end{align}
locally for a gauge transformation $g_\sigma$. By the result in the smooth case (this is an application of the Frobenius theorem)  a  solution $g_{\sigma}$ exists and by multiplying by a constant matrix we may assume without loss of generality that $g_{\sigma}(x_0)=id$. This makes the solution unique and thus if a cell $\tau$ is a face of a cell $\sigma$ then, since $A^{\sigma} \big|_{\tau}=A^{\tau}$, it must be $g_{\sigma} \big|_{\tau}=g_{\tau}$. It follows that the new frames $s_\sigma=g_\sigma \circ  s_\sigma^0$ patch together to define a piecewise smooth frame $s$ in a neighborhood of $x_0$ in $X$. The flat structure is now defined by the local frames $\{ s\}$.
\end{proof}

\begin{definition}\label{def:holonomy0}
A section $s \in \Omega^0(X, E)$ is \emph{parallel with respect to $D$} if $Ds = 0$. 
Given a  $C^1$ curve $c : [a, b] \rightarrow X$, a section $s$ is \emph{parallel along $c$ with respect to $D$} if $D_{c'(t)} s = 0$.
Given a  $C^1$ curve $c : [a, b] \rightarrow X$ and $s_a \in \pi^{-1}(c(a))$ the \emph{parallel transport of $s$ along $c$ with respect to $D$} is the section $s : \pi^{-1} \left(c([a, b]) \right) \rightarrow E$ which is given locally by the solution to the equation
$$
\frac{d s(c(t))}{dt} + A_{c(t)} (c'(t)) s(c(t)) = 0 .
$$
\end{definition}

\begin{lemma}\label{lem:holonomyindep}
Let $c_1, c_2  : [a, b] \rightarrow X$ be two closed  $C^1$ curves in $X$ which are homotopy equivalent, and which satisfy $x_0 = c_1(a) = c_1(b) = c_2(a) = c_2(b)$. Let $D$ be a smooth flat connection on a rank $r$ bundle $\pi : E \rightarrow X$, and let $s_1$ and $s_2$ be the parallel transport with respect to $D$ along $c_1$ and $c_2$ respectively, with initial condition $s_0 \in \pi^{-1}(x_0)$.  If $F_D = 0$ then $s_1(c_1(b)) = s_2(c_2(b))$.
\end{lemma}
\begin{proof}
As usual note that it suffices to show that the holonomy is trivial around a homotopically trivial loop.
If there is a homotopy equivalence between two loops that is constant except on a single cell, then standard theorems for smooth manifolds show that the holonomy around the two loops is the same. 
Given a homotopically trivial loop $\gamma$, there is a sequence of homotopy equivalences $\gamma \sim \gamma_1$, $\gamma_1 \sim \gamma_2$, $\ldots$, $\gamma_N \sim \id$ between $\gamma$ and the trivial loop (denoted $\id$), such that each homotopy equivalence is constant except  on a single $n$-cell. For example, one can do this by identifying the fundamental group with the edge group of a simplicial complex (cf. \cite{armstrong}, Section 6.4). Therefore, the holonomy of $\gamma$ is the same as the holonomy of each $\gamma_n$ along this sequence of homotopy equivalences, and so  the holonomy of $\gamma$ is trivial.
\end{proof}

\begin{definition}
A flat connection $D$ on a rank $r$ vector bundle $\pi:E \rightarrow X$ defines a representation $\rho : \pi_1(X) \rightarrow \GL(r, \C)$ called the \emph{holonomy representation of $D$}.  A flat connection is called \emph{irreducible} if its holonomy representation is irreducible. The space of irreducible, flat smooth connections is denoted $\mathcal{A}^{\C, irr}(E)$.
\end{definition}

\begin{lemma}\label{lem:representation-flat-connection}
A representation $\rho: \pi_1(X) \rightarrow \GL(r, \C)$ defines a flat connection on a bundle $\pi : E_{\rho} \rightarrow X$ with holonomy representation  $\rho$. Moreover, the flat connection on $E_\rho$ depends continuously on the representation $\rho$.
\end{lemma}

\begin{proof}
In the standard way, from a representation $\rho : \pi_1(X) \rightarrow \GL(r, \C)$ we construct a flat vector bundle $E_{\rho} \rightarrow X$, with total space
\begin{equation}
E_{\rho} = \tilde{X} \times_\rho \C^r,
\end{equation}
where $\tilde{X}$ is the universal cover of $X$, and the equivalence is by deck transformations on the left factor $\tilde{X}$, and via the representation $\rho$ on the right factor $\C^r$. On each trivialisation we have the trivial connection $d$, and since the transition functions of $E$ are constant, then this connection is globally defined. Since the deck transformations depend continuously on the representation $\rho$ then the flat connection on $E_\rho$ depends continuously on $\rho$.
\end{proof}

\begin{corollary}\label{cor:trivialization-flat-connection}
A flat connection on a vector bundle over a simply connected complex $X$ is complex gauge equivalent to the trivial connection $d$ on the trivial vector bundle.
\end{corollary}

\begin{definition}
The \emph{$\SL(r, \C)$ character variety} is the space of irreducible representations $\rho : \pi_1(X) \rightarrow \SL(r, \C)$ modulo conjugation by $\SL(r, \C)$
\begin{equation}
\mathcal{M}_{char} = \left\{ \text{irreducible reps }\, \rho : \pi_1(X) \rightarrow \SL(r, \C) \right\} / \SL(r, \C)
\end{equation}
\end{definition}

The next Lemma is a trivial consequence of the path lifting property for principal bundles (cf. for example \cite{goldman-hirsch}, Lemma 3).

\begin{lemma}\label{lemma:components} Two characters defined by the representations $\rho$ and $\rho'$ belong to the same connected component of $\mathcal{M}_{char}$ if and only if the vector bundles $E_\rho$ and $E_{\rho'}$ are  smoothly isomorphic.
\end{lemma}

In view of the above, let $\mathcal{C}$ denote the set of connected components of $ \mathcal{M}_{char}$. Then we can write 
\[
 \mathcal{M}_{char}= \bigsqcup_{c \in \mathcal{C}} \mathcal{M}_{char}^c
\]
and write $E_c=E_\rho$ for any representative in the isomorphism class of bundles defined by  $\rho \in  \mathcal{M}_{char}^c$.

\begin{remark} Since we are interested in the $\SL(r, \C)$ character variety instead of the $\GL(r, \C)$ character variety we need to fix determinants in our definitions of connections and gauge transformations. Henceforth \emph{ we will impose the condition that all connection forms are traceless and all gauge transformations have determinant one.} For the sake of notational simplicity  we will keep the same notation as before for the various spaces of  $\SL(r, \C)$ connections and gauge groups.
\end{remark}

\begin{proposition}\label{prop:holonomyequivalence}
$$\mathcal{A}_{flat}^{\C, irr} (E_c) / \mathcal{G}^\C (E_c)  \cong \mathcal{M}_{char}^c$$
\end{proposition}

\begin{proof}
The holonomy map applied to an irreducible flat connection $D$ gives an irreducible representation $\rho : \pi_1(X, x_0) \rightarrow \GL(r, \C)$. The action of a complex gauge transformation $g \in \mathcal{G}^\C (E_c) $ on $D$ induces the conjugate action of an element $\xi = g(x_0) \in \GL(r, \C)$ on $\rho$. Therefore we have a continuous map $\tau : \mathcal{A}_{flat}^{\C, irr} (E_c)/ \mathcal{G}^\C (E_c)  \rightarrow \mathcal{M}_{char}^c$. Note that $\tau([D_1]) = \tau([D_2])$ implies that the flat structures associated to $D_1$ and $D_2$ by Theorem \ref{thm:flat-connection-flat-bundle} are complex gauge-equivalent, and so $D_1$ and $D_2$ are complex gauge-equivalent. Therefore $\tau$ is injective.

Similarly, given a representation $\rho : \pi_1(X, x_0) \rightarrow \GL(r, \C)$ we construct a flat connection $d$ on the flat bundle $E_\rho$ as in the proof of Lemma \ref{lem:representation-flat-connection}. If we conjugate the representation by an element $\xi \in \GL(r, \C)$, then the flat connection associated to this new representation is related to $E_\rho$ by a global change of co-ordinates using the action of $\xi$ on the fibres of $E_\rho$. Therefore the two flat bundles are complex gauge-equivalent, and so conjugate representations give $\mathcal{G}^\C (E_c) $-equivalent flat connections, which gives us a continuous map $\zeta : \mathcal{M}_{char} \rightarrow \mathcal{A}_{flat}^{\C, irr}(E_c) / \mathcal{G}^\C (E_c) $.
Lemma \ref{lem:representation-flat-connection} shows that $\tau \circ \zeta = \id$. Since $\tau$ is injective then this implies that $\zeta \circ \tau = \id$ and so $\tau$ is a homeomorphism $\mathcal{A}_{flat}^{\C, irr} (E_c) / \mathcal{G}^\C (E_c)  \cong \mathcal{M}_{char}^c$.
\end{proof}

\subsection{Relationship to Higgs bundles}\label{subsec:relate-higgs}

Given a complex $X$ with universal cover $\tilde{X}$, fix an irreducible representation $\rho : \pi_1(X) \rightarrow \SL(r, \C)$, and let $E = \tilde{X} \times_\rho \C^r \rightarrow X$ be as before. We also fix  a $\rho$-equivariant map $u : \tilde{X} \rightarrow \SL(r, \C) / \SU(r)$, locally Lipschitz away from the 0-skeleton $X^0$ of $X$. We now recall the basic construction from \cite{Corlette88} and  \cite{Donaldson87}.

\begin{enumerate}

\item \label{item:positive} The complexified tangent space $T^{\C}_h(\SL(r, \C) / \SU(r))$ can be identified (independent of $h$) with the space of traceless matrices  and this gives a trivialization of the complexified tangent bundle $T^{\C}(\SL(r, \C) / \SU(r)) \cong \SL(r, \C) / \SU(r) \times \mathfrak {sl} (r, \C)$.  

\item \label{item:L-C-connection} In the trivialization given in (\ref{item:positive}) the Levi-Civita connection at a point $h \in \SL(r, \C) / \SU(r)$ has the form
\begin{equation*}
\nabla_X Y = dY(X)- \frac{1}{2} \left( dh(X)h^{-1}Y+ Yh^{-1}dh(X) \right) ,
\end{equation*}
where we use the notation $h$ to indicate  left translation by $h$.

\item \label{item:flat-connection} 
The identification $  h^{-1} \big(T^{\C}_h (\SL(r, \C) / \SU(r) ) \big)  \cong T^{\C}_{\id} (\SL(r, \C) / \SU(r) )\cong  \mathfrak {sl} (r,\C)$ gives another isomorphism $ \theta:  T^{\C} (\SL(r, \C) / \SU(r))  \rightarrow \SL(r, \C) / \SU(r) \times  \mathfrak {sl} (r,\C)$.  It follows immediately from (\ref{item:L-C-connection}) that in the coordinates given by $\theta$, the  Levi-Civita connection is given by 
\begin{eqnarray*}
\nabla_X s &=&  h^{-1}\nabla_X (hs) \\
&=& ds(X) + \frac{1}{2} \left[ h^{-1} dh(X), s \right].
\end{eqnarray*}
We thus conclude that in the above coordinates 
\begin{equation}\label{eqn:fund-eqn0}
d=\nabla +\frac{1}{2} \left[  h^{-1}dh, . \right].
\end{equation}
\item \label{item:action}The  isomorphism $\theta$ is equivariant with respect to the $\PSL(r, \C)$ action on the complexified tangent bundle $T^{\C}(\SL(r, \C) / \SU(r))$ and the adjoint representation on $T^{\C}_{id}(\SL(r, \C) / \SU(r)) \cong \mathfrak sl(r, \C)$. 

\item  \label{item:the equations} Given $u$ as above, consider the pull-backs $\mathcal D = u^*d$ and $ d_A = u^*\nabla $ on the trivial bundle $\tilde X \times T^{\C}_{id}(\SL(r, \C) / \SU(r))\cong  \tilde X \times \mathfrak sl(r, \C)$. First notice, that since $u^*d$ is trivial and $u$ is $\rho$-equivariant $\mathcal D $ descends to a flat connection of holonomy $\rho$ on $E_\rho$. Again, by the  $\rho$-equivariance of $u$ and (\ref{item:action}) the connection  $ d_A$  descends to a connection on  $ad(E_\rho)$  over $X$. Moreover, since its connection form acts by the adjoint representation it defines an $\SL(r,\C)$ connection on the bundle $E_\rho$ over $X$ and (\ref{eqn:fund-eqn0}) implies
\begin{equation} \label{eqn:fund-eqn}
\mathcal D=d_A +\psi,  \  \ \psi= \frac{1}{2}  u^{-1}du.
\end{equation}
Since $\mathcal D$ is a flat connection, then
\begin{align} 
F_A + \psi \wedge \psi & = 0 \label{eqn:mu1'} \\
d_A \psi & = 0 \label{eqn:mu2'}.
\end{align}
\end{enumerate}

\subsection{The balancing condition}\label{subsec:balancing}

\begin{definition}\label{def:balancing}
A  smooth $1$-form $\omega=\{\omega_{\sigma}\} \in \Omega^1(X)$ satisfies the \emph{balancing condition} if for every $(n-1)$-cell $\tau$, we have
\begin{equation}\label{eqn:balancing}
\sum_{\sigma>\tau} \omega_{\sigma}(e^n_{\sigma}) = 0
\end{equation}
where $\sigma>\tau$ implies that $\tau$ is a face of $\sigma$, and $e^n_{\sigma}$ is an inward pointing normal vector field along $\tau$ in $\sigma$.   The set $\Omega^1_{bal}(X)$ is elements of $\Omega^1(X)$ satisfying the balancing condition.  
\end{definition}

\begin{definition}
Let $E$ be  smooth vector bundle on $X$ of rank $r$ and let $p: \tilde X \rightarrow X $ be the universal cover. We assume that the pullback bundle $p^*(E)$ over 
$\tilde X$ is trivial \emph{with a fixed trivialization} $p^*(E) \cong \tilde X \times \C^r$ (if the connection is flat then this is always valid by Corollary \ref{cor:trivialization-flat-connection}).  A connection $D \in \mathcal{A}^{\C, irr}(E)$  is called \emph{balanced} if its pullback $p^*(D)$ to $p^*(E)$ can be written (in the given trivialization) as
$p^*(D)=d+A$ where all the components$A_{ij} \in \Omega^1_{bal}(\tilde X)$.
Let $\mathcal{A}_{bal}^{\C, irr}(E)$ be the space of irreducible, smooth, balanced $\GL(r, \C)$ connections, and let $\mathcal{A}_{bal}^{irr}(E)$ denote the space of irreducible, smooth, balanced connections compatible with the Hermitian metric $h$ on $E$. In what follows, if the meaning is clear then the notation for the metric is suppressed. 
\end{definition}

\begin{definition}
Let $E$ be as in the previous Definition. Given $g \in \mathcal{G}^\C (E) $, let $\tilde g$ denote the induced gauge transformation of $p^*(E)$. We define $\mathcal{G}^\C_{bal}(E)$ (resp. $\mathcal{G}_{bal}(E)$)  to be the group of complex (resp. unitary) gauge transformations such that $g \in \mathcal{G}^\C_{bal}(E)$ (resp. $g \in \mathcal{G}_{bal}(E)$) implies that $d \tilde g_{ij} \in \Omega^1_{bal}(\tilde X)$.
\end{definition}

\begin{remark}
Note that via \eqref{eqn:group-action}, the group $\mathcal{G}_{bal}^\C(E)$ acts on the space $\mathcal{A}_{bal}^{\C, irr}(E)$, and $\mathcal{G}_{bal}(E)$ acts on $\mathcal{A}_{bal}^{irr}(E)$.
\end{remark}

\begin{remark}\label{rem:makessense}
In this paper we are interested in flat bundles. Note that Corollary~\ref{cor:trivialization-flat-connection} implies that the pullback of a flat bundle to the universal cover is trivial. By choosing a trivialization it thus makes sense to talk about balanced connections and gauge transformations.
\end{remark}

\section{Harmonic maps and Higgs bundles} \label{hmhb}
In this section we describe the relationship between Higgs bundles and harmonic maps from a  complex $X$ into the space $\SL(n, \C) / \SU(n)$, a generalisation  of the construction of \cite{Donaldson87} and \cite{Corlette88}. 
\emph{From now on $X$ will denote an admissible 2-dimensional simplicial complex without boundary.}   We will further assume that $X$ is  equipped  with a Riemannian metric $g_{\sigma}$ on each simplex $\sigma$ of $X$ with the following conditions:  (i) for any 2-cell $\sigma$, $(\sigma, g_{\sigma})$ is isometric to an interior of an equilateral triangle in ${\bf R}^2$ and for any 1-cell $\tau$, $(\tau,g_{\tau})$ is isometric with the open unit interval in ${\bf R}$, and (ii)  given a 2-cell $\sigma$ and a 1-cell $\tau$ that is a face of $\sigma$, we have $g_{\sigma} \big|_{\tau}=g_{\tau}$.  It is not hard to extend the results of this section to general Riemannian metrics and also general 2-dimensional cell complexes. We endow $\SL(n,\C)/\SU(n)$ with a Riemannian metric of non-positive sectional curvature such that $\SL(n,\C)$ acts by isometries.

\subsection{Estimates of harmonic maps}

\begin{theorem}
Let $X$ be a 2-complex as before with universal cover $\tilde{X}$ and $\rho:\pi_1(X) \rightarrow \SL(n,\C)$ be an irreducible representation.  Then there exists a unique $\rho$-equivariant harmonic map $u:\tilde{X} \rightarrow Y:=\SL(n,\C)/\SU(n)$.
\end{theorem}

\begin{proof}
The existence is a special case of Theorem 4.5 of \cite{DaskalMese06}.  Uniqueness follows from \cite{Mese02}.
\end{proof}

Let $p$ be a vertex (i.e. 0-cell)  of  $X$.   Define ${\mathcal S}_1(p)$ be the set of 1-cells  of $X$ containing  $p$ in its closure.  Let $\tau$ be a 1-cell of $X$.  Define ${\mathcal S}_2(\tau)$ be the set of 2-cells of $X$ containing $\tau$ in its closure.  For $\tau \in {\mathcal S}_1(p)$ and $\sigma \in {\mathcal S}_2(\tau)$, we define polar coordinates $(r,\theta)$ of $\sigma \cup \tau$ centered at $p$ by setting $r$ to be the distance  from $p$ to a point $p_0 \in \sigma \cup \tau$ and $\theta$ to be the angle between $\tau$ and the line $\overline{pp_0}$ connecting $p$ and $p_0$.

The next theorem is one of the main technical results of the paper and describes the singular behavior of harmonic maps near the lower dimensional strata.  

\begin{theorem} \label{thm:harmonic-map-existence}
If $u:X \rightarrow Y$ is a harmonic map, then $u \in C^{\infty}(\sigma \cup \tau)$ for any one-cell $\tau$ and two-cell $\sigma$ so that $\sigma \in {\mathcal S}_2(\tau)$.  (In other words, the restriction of $u$ to $\sigma$ is $C^{\infty}$ up to $\tau$).  Moreover, if $(r,\theta)$ are the polar coordinates of $\sigma \cup \tau$ centered at a 0-cell $p$  and $u$ is given in local coordinates $(u^1, \dots, u^M)$ in a neighbourhood of  $u(p)$, we have the following derivative bounds for $u^m$ in a neighbourhood of $p$:

\[
\left| \frac{\partial u^m}{\partial r}  \right| \leq Cr^{\alpha-1}, \ \ \left| \frac{\partial u^m}{\partial \theta}  \right|  \leq Cr^{\alpha}
\]
\[
\left| \frac{\partial^2 u^m}{\partial r^2}  \right| \leq Cr^{\alpha-2}, \ \ \left| \frac{\partial^2 u^m}{\partial r \partial \theta}  \right| \leq Cr^{\alpha-1}, \ \ \left| \frac{\partial^2 u^m}{\partial \theta^2}  \right| \leq Cr^{\alpha}
\]
\[
\left| \frac{\partial^3 u^m}{\partial r^3}  \right| \leq Cr^{\alpha-3}, \ \ \left| \frac{\partial^3 u^m}{\partial^2 r \partial \theta}  \right| \leq Cr^{\alpha-2}, \ \ \left| \frac{\partial^3 u^m}{\partial r \partial^2 \theta}  \right| \leq Cr^{\alpha-1}, \ \ \left| \frac{\partial^3 u^m}{\partial \theta^3}  \right| \leq Cr^{\alpha}
\]
for some constants $C>0$ and $\alpha>0$ depending on the total energy of $u$ and the geometry of the complex $X$. Furthermore,  $\alpha$ can be chosen independently of the choice of the 0-cell $p$ of $X$.  
\end{theorem}

\begin{proof}
Let $\sigma=\sigma_1, \dots, \sigma_J$ be the 2-cells in ${\mathcal S}_2(\tau)$.  For each $j=1, \dots, J$, we let $(x,y)$ be the Euclidean coordinates of $\overline{\sigma_j \cup \tau}$ so that (i) $p$ is given as $(x,y)=(0,0)$, (ii) if $(x,y) \in \tau$ then $x>0$ and $y=0$ and (iii) if $(x,y) \in \sigma_j$ then $x,y>0$.   Let $u^m_j=u^m\big|_{\sigma_j}$.  The fact that $u^m_j \in C^{\infty}(\sigma_j \cup \tau)$ follows from Theorem 4 and Corollary 6 of \cite{DaskalMese08}.  

We will now compute the first derivative bounds with respect to the polar coordinates $r$ and $\theta$.  By Theorem 6.2  of \cite{DaskalMese06}, we have the inequality 
\[
|\nabla u|^2(r,\theta) \leq C r^{2\alpha-2}
\]
for some $\alpha>0$.  More specifically,  $\alpha$ can be chosen to be the order of $u$ at $p$, i.e. 
\[
\alpha = \lim_{r \rightarrow 0} \frac{r {\displaystyle  \int_{B_r(p)} } |\nabla u|^2 d\mu}{ {\displaystyle \int_{\partial B_r(p)} }d^2(u,u(p)) ds}.  
\]
Hence,
\begin{equation} \label{xy}
\left|\frac{\partial u^m_j}{\partial x} \right|  \leq Cr^{\alpha-1}\  \mbox{ and }\ \left|\frac{\partial u^m_j}{\partial y} \right| \leq Cr^{\alpha-1}.
\end{equation}
Using the fact that $x=r\cos \theta$ and $y=r\sin \theta$, we get
\[
\frac{\partial u^m_j}{\partial r}=\frac{\partial u^m_j}{\partial x} \cos \theta+\frac{\partial u^m_j}{\partial y} \sin \theta \ \mbox{ and } \
 \frac{\partial u^m_j}{\partial \theta}=-\frac{\partial u^m_j}{\partial x} r\sin \theta + \frac{\partial u^m_j}{\partial y} r \cos \theta.
\]
This  immediately implies
\[
\left|\frac{\partial u^m_j}{\partial r} \right| \leq Cr^{\alpha-1} \ \mbox{ and } \ \left|\frac{\partial u^m_j}{\partial \theta} \right| \leq C r^{\alpha}.
\]
We will now establish the second derivative estimates of $u^m_j$ for a points $(r,\theta)$ on $\sigma_j \cup \tau$ with $\theta$ sufficiently small.    We will need the following notations: for a function $\varphi$ and a domain $\Omega \subset \mathbb{R}^2$, we set
\[
|\varphi|_{0;\Omega} = \sup_{p \in \Omega} |\varphi(p)|
\]
\[
|D \varphi|_{0;\Omega}=\sup_{p \in \Omega} \max \left\{ \left|\frac{\partial \varphi}{\partial x}(p) \right|,  \left|\frac{\partial \varphi}{\partial y}(p)\right|  \right\}
\]
\[
|D^2 \varphi|_{0;\Omega}=\sup_{p \in \Omega} \max \left\{ \left|\frac{\partial^2 \varphi}{\partial x^2}(p) \right|,   \left|\frac{\partial^2 \varphi}{\partial x \partial y}(p)\right|, \left|\frac{\partial^2 \varphi}{\partial y^2}(p)\right|  \right\}
\]
\[
[\varphi]_{\beta;\Omega} =\sup_{p,q \in \Omega, p\neq q} \frac{|\varphi(p)-\varphi(q)|}{|p-q|^{\beta}} 
\]
\[
[D \varphi]_{\beta;\Omega}=\sup_{p,q \in \Omega, p\neq q} \frac{1}{|p-q|^{\beta}} \max \left\{ \left|\frac{\partial \varphi}{\partial x}(p)-\frac{\partial \varphi}{\partial x}(q) \right|,  \left|\frac{\partial \varphi}{\partial y}(p)-\frac{\partial \varphi}{\partial y}(q)\right|  \right\}
\]
\[
[D^2 \varphi]_{\beta;\Omega}=\sup_{p,q \in \Omega, p\neq q} \frac{1}{|p-q|^{\beta}} \max \left\{ \left|\frac{\partial^2 \varphi}{\partial x^2}(p)-\frac{\partial^2 \varphi}{\partial x^2}(q) \right|,  \left|\frac{\partial^2 \varphi}{\partial x \partial y}(p)-\frac{\partial^2 \varphi}{\partial x \partial y}(q) \right|,  \left|\frac{\partial^2 \varphi}{\partial y^2}(p)-\frac{\partial^2 \varphi}{\partial y^2}(q)\right|  \right\}.
\]
Let  $T :=\{(x,y) \in \mathbb{R}^2: y\geq0, y \leq \sqrt{3}x, y \leq -\sqrt{3}x +\sqrt{3}\}$, $T^-=\{(x,-y) \in \mathbb{R}^2:(x,y) \in T\}$ and $\hat{T}=T \cup T^-$. Fix $m$ and $j$ and define $U:\hat{T} \rightarrow \mathbb{R}$ by setting
\[
U(x,y) = \left\{ 
\begin{array}{ll}
u^m_j(x,y) & \mbox{if }  y\geq 0\\
-u^m_j(x,-y)+ {\displaystyle{ \frac{2}{J} \sum_{j'=1}^J}} u^m_{j'} (x,-y) & \mbox{if } y<0.
\end{array}
\right.
\]
Let 
\begin{equation}\label{eqn:Christoffel}
\Gamma^m_j=\sum_{p,q=1}^M\Gamma^m_{pq}(u_j)\left( \frac{\partial u^p_j}{\partial x} \frac{\partial u^q_j}{\partial x} + \frac{\partial u^p_j}{\partial y} \frac{\partial u^q_j}{\partial y}  \right)
\end{equation}
where $\Gamma^m_{pq}$ are the Christoffel symbols of $Y$ with respect to the local coordinates $(u^1, \dots, u^M)$. Since
the harmonic map equation
\[
\triangle u^m_j=\Gamma^m_j
\]
is satisfied in $T$, 
if we set
\[
f(x,y)=\left\{
\begin{array}{ll}
\Gamma^m_j(x,y) & \mbox{if } y \geq 0\\
-\Gamma^m_j(x,-y)+\frac{2}{J} {\displaystyle \sum_{j'=1}^J} \Gamma^m_{j'}(x,-y) & \mbox{if } y <0,
\end{array}
\right.
\]
then $U$ satisfies the Poisson equation
\[
\triangle U = f
\]
in $\hat{T}$.    If $B_{2R}(p) \subset \hat{T}$, then
elliptic regularity theory (cf.  \cite{GilbargTrudinger83} or Lemma 3 on p13 of \cite{Simon96}) implies
\[
R^{1+\beta} [DU]_{\beta;B_{\frac{3R}{2}}(p)} \leq  C(|U|_{0;B_{2R}(p)}+R^2|f|_{0;B_{2R}(p)}).
\]
If we choose $p$ to be a point on $\tau$ and $R$ to be the largest number so that $B_{2R}(p) \subset \hat{T}$, then $R$ is proportional to $r$ where $r$ is the distance of $p$ to the vertex $p$.  Furthermore, the distance from $p$ to any point of $B_{2R}(p)$ is bounded uniformly by some constant multiple of $r$.  Hence, assuming $U(0,0)=0$ without a loss of generality, we have
\[
[DU]_{\beta;B_{\frac{3R}{2}}(p)} \leq C(r^{-1-\beta}|U|_{0;B_{2R}(p)}+r^{1-\beta}|f|_{0;B_{2R}(p)}) \leq C(r^{-1-\beta+\alpha}+r^{-\beta+2\alpha-1}) \leq Cr^{-\beta+\alpha-1}.
\]
Here, we have used the H\"{o}lder continuity of $u^m_j$ (hence of $U$) near $p$ with H\"{o}lder exponent $\alpha$ (cf. Theorem 3.7 of \cite{DaskalMese06}) and the inequalities of (\ref{xy}) along with the fact that $f$ is quadratic in $Du^m_j$ from \eqref{eqn:Christoffel}.  
Thus, with $B_{\frac{3R}{2}}^+(p)=B_{\frac{3R}{2}}(p) \cap\{y \geq 0\}$, we obtain
\[
[Du^m_j]_{\beta;B_{\frac{3R}{2}}^+(p)} \leq  Cr^{-\beta+\alpha-1}.
\]
Since $u$ is Lipschitz, then this equation along with \eqref{xy} and \eqref{eqn:Christoffel} implies that
\begin{equation} \label{1beta}
[\Gamma^m_j]_{\beta;B_{\frac{3R}{2}}^+(p)} \leq  C|Du^k_j|_{0;B_{\frac{3R}{2}}^+(p)} [Du^\ell_j]_{\beta;B_{\frac{3R}{2}}^+(p)} \leq Cr^{-\beta+2\alpha-2}.
\end{equation}
We are now ready to prove the second derivative bounds of $u^m_j$.  Note that we have the set of  partial differential equations 
\begin{eqnarray}\label{joel1}
\triangle u^m_j=\Gamma^m_j &  j=1, \dots, J & m=1, \dots M
\end{eqnarray}
in $T$, along with boundary conditions
\begin{eqnarray}
u^m_j-u^m_1=0 &   j=2, \dots, J &  m=1, \dots, M \label{joel2} \\
\sum_{j=1}^J \frac{\partial u_j^m}{\partial y} =0 &   m=1, \dots, M \label{joel3}
\end{eqnarray}
in $B=\{ (x,y) \in \mathbb{R}^2: y=0, 0<x<1\}$.  
This  is a system of $JM$ number of equations containing $JM$ number of unknowns (i.e. $u^m_j$) along with $JM$ number of boundary conditions.  If we assign weights $s_j^m=0$ to the equations, weights $t_j^m=2$ to the unknowns, weights $r_j^m=-2$ for $j=2, \dots, M$ and $r^m_1=-1$ to the boundary conditions, then this system  is said to be elliptic with complementing boundary condition according  to the the elliptic regularity theory of \cite{AgmonDouglisNirenberg64} (or elliptic and coercive in \cite{KNS78}).   Hence, we have the Schauder estimates (cf. Theorem 9.1 of \cite{AgmonDouglisNirenberg64}), 
\begin{eqnarray*}
\lefteqn{ R^2 |D^2u^m_j|_{0;B^+_R(p) }+R^{2+\beta} [D^2u^m_j]_{\beta; B^+_R(p) }  }\\
& \leq & C(|\Gamma^m_j|_{0;B^+_{\frac{3R}{2}}(p)}+ R^{2+\beta} [\Gamma^m_j]_{\beta; B^+_{\frac{3R}{2}}(p)} + |u^m_j|_{0;B^+_{\frac{3R}{2}}(p)}).
\end{eqnarray*}
With the same choice of $p$ and $R$ as above, we obtain
\[
|D^2u^m_j|_{0;B^+_R(p) }
\leq C(|\Gamma^m_j|_{0;B^+_{\frac{3R}{2}}(p)}+ r^{\beta} [\Gamma^m_j]_{\beta; B^+_{\frac{3R}{2}}(p)} + r^{-2}|u^m_j|_{0;B^+_{\frac{3R}{2}}(p)}).
\]
The above inequality, along with (\ref{1beta}), implies
\[
|D^2u^m_j|_{0;B^+_R(p) }
\leq C(r^{2\alpha-2}+r^{2\alpha-2}+r^{\alpha-2}) \leq Cr^{\alpha-2}.
\]
Since
\[
\frac{\partial^2 u^m_j}{\partial r^2}=\frac{\partial^2 u^m_j}{\partial x^2} \cos^2 \theta + 2\frac{\partial^2 u^m_j}{\partial x \partial y} \sin \theta \cos \theta+\frac{\partial^2 u^m_j}{\partial y^2} \sin^2 \theta,
\]
\begin{eqnarray*}
\frac{\partial^2 u^m_j}{\partial r \partial \theta} & = & -\frac{\partial^2 u^m_j}{\partial x^2} r \sin \theta \cos \theta + \frac{\partial^2 u^m_j}{\partial x \partial y} r \cos^2 \theta - \frac{\partial u^m_j}{\partial x} \sin \theta \\
& & - \frac{\partial^2 u^m_j}{\partial x \partial y} r \sin^2 \theta + \frac{\partial^2 u^m_j}{\partial y^2} r \sin \theta \cos \theta + \frac{\partial u^m_j}{\partial y} \cos \theta,
\end{eqnarray*}
\begin{eqnarray*}
\frac{\partial^2 u^m_j}{\partial \theta^2} & = & \frac{\partial^2 u^m_j}{\partial x^2} r^2 \sin^2 \theta+ 2\frac{\partial^2 u^m_j}{\partial x \partial y} r^2  \sin^2 \theta + \frac{\partial^2 u^m_j}{\partial y^2} r^2 \cos \theta\\
& = & -\frac{\partial u^m_j}{\partial x}  r \cos \theta -\frac{\partial u^m_j}{\partial y}  r\sin \theta,
\end{eqnarray*}
we immediately obtain
\[
\left|\frac{\partial^2 u^m_j}{\partial r^2} \right|\leq C r^{\alpha-2}, \left|\frac{\partial^2 u^m_j}{\partial r \partial \theta} \right| \leq Cr^{\alpha-1} \mbox{ and } \left|\frac{\partial^2 u^m_j}{\partial \theta^2} \right| \leq Cr^{\alpha}
\]
at $(r,\theta)$ for $\theta$ sufficiently small.  This restriction on $\theta$  is due to  the choice of $R$ and $p$.    For $(r,\theta)$  with $\theta$ sufficiently large, we can use a similar argument using standard elliptic regularity theory (e.g. \cite{GilbargTrudinger83} or Lemma 3 on p13 of \cite{Simon96}) in the interior of $\sigma$.  
The third derivative estimates follow the same way from the first two by bootstrapping the elliptic equations (\ref{joel1}) with boundary conditions (\ref{joel2}) and (\ref{joel3}).

Section 4 of \cite{DaskalMese08} shows the that order of $u$ at a 0-cell $p$ can be bounded from below by 2$\lambda_v^{comb}$ where $\lambda_v^{comb}$ is the combinatorial eigenvalue of the link of $v$ which is always a positive quantity.  Hence choosing $\alpha$ to be the minimum of $2\lambda_v^{comb}$ over all 0-cells  of $X$, we have established the last assertion of the Theorem.
\end{proof}

\subsection{Weighted Sobolev spaces}\label{subsec:weightedsobolev}

In this subsection we recall the important features of the weighted Sobolev spaces used in this paper. The main references are \cite{Adams75}, \cite{ChoquetChristo81}, \cite{DaskalWentworth97} and \cite{Taubes87}.
In the following we fix a  smooth vector bundle $E$ of rank $r$ over a $2$-complex $X$ with a Hermitian metric, and a fixed Riemannian metric on the base space $X$. Define the space $C_0^\infty(E)$ to be the space of smooth sections $s \in \Omega^0(X, E)$ that satisfy $s(p) = 0$ whenever $p$ is a vertex of $X$. In the local model $\tilde{B}(r)$ around each vertex $p$, we define local co-ordinates $(t, \theta) = (- \log r, \theta)$, where $(r, \theta)$ are the standard polar co-ordinates in a neighbourhood of the vertex $p$. 
To define a norm on $C_0^\infty(E)$, choose open neighbourhoods $U_{x_i}$ for each vertex $x_i$, and cover the rest of $X$ with neighbourhoods $V_\alpha$ that do not intersect any of the vertices. For $\delta \in \R$, the space $L_\delta^p$ is the completion of $C_0^\infty(E)$ in the norm
\begin{equation}
\| s \|_{L_\delta^p} = \left( \sum_i  \int_{U_{x_i}} e^{t \delta} | s |^p  + \sum_\alpha \int_{V_\alpha} |s|^p \right)^{1/p}
\end{equation}
where we use $e^{t \delta}$ to denote the co-ordinates in a neighbourhood of a vertex. Away from all of the vertices, $e^{t \delta}$ is bounded and $s$ is continuous, and so the question of whether the norm $\| \cdot \|_{L_\delta^p}$ is finite only depends on the choice of co-ordinates near each vertex. Different choices of $V_\alpha$ will lead to equivalent norms.

Given a vertex $p$, we say that a connection of the form $\nabla = d $ is in \emph{temporal gauge} in a neighbourhood of $p$. Given a fixed connection $\nabla_0$ in temporal gauge, and a positive integer $k$, we define the \emph{weighted Sobolev space} $L_{k, \delta}^q(E)$ as the completion of $C_0^\infty(E)$ in the norm
\begin{equation}
\| s \|_{L_{k, \delta}^q} = \sum_{\ell = 0}^k  \| \nabla_0^\ell s \|_{L_\delta^q} 
\end{equation}
Note that in this paper we are considering bundles with a fixed trivialization on the universal cover (cf. Remark~\ref{rem:makessense}). Since the star of a vertex $p$ in $X$ is simply connected it follows that we have a fixed trivialization of $E$ in a neighbourhood of $p$. It thus makes sense to talk about connections on $E$ in temporal gauge.

It is a standard fact that the spaces $L_{k, \delta}^q$ do not change if we either: (a) change the connection $\nabla_0$ outside a neighbourhood of the vertices of $X$, or (b) change the co-ordinates outside a neighbourhood of the vertices.
Lemma 2.5 of \cite{ChoquetChristo81} states that the usual multiplication theorems for Sobolev spaces on compact manifolds carry over to the weighted Sobolev spaces studied here. To be more precise, we have that the multiplication map $L_{s_1, \delta_1}^2 \times L_{s_2, \delta_2}^2 \rightarrow L_{s, \delta}^2$ is continuous if $s_1, s_2 \geq s$, $s < s_1 + s_2 - n/2$ and $\delta < \delta_1 + \delta_2 + n/2$, where $n$ is the dimension of the complex $X$.

Following Section 3.1 of \cite{DaskalWentworth97} we define the space of weighted connections $\mathcal{A}_\delta^\C(E)$ to be the space of all connections whose connection form is an element of $L_{1, \delta}^2$, and the space $\mathcal{A}_\delta (E) \subset \mathcal{A}_\delta^\C(E)$ to be the subset of all unitary connections. The weighted gauge group $\mathcal{G}_\delta (E)$ is defined as follows. Let $\nabla_0$ be a connection in temporal gauge and define
\begin{equation}
\mathcal{R} = \left\{ v \in L_{2, loc}^2( \End(E)) \, : \, \| \nabla_0 v \|_{L_{1, \delta}^2} < \infty \right\}
\end{equation}
Then the \emph{weighted gauge group} is defined as
\begin{equation}
\mathcal{G}_\delta (E) = \left\{ v \in \mathcal{R} \, : \, v v^* = \id, \det v = 1 \right\}.
\end{equation}
and the complexified gauge group is
\begin{equation}\label{eqn:weighted-complex-gauge-group}
\mathcal{G}_\delta^{\C} (E)= \left\{ v \in \mathcal{R} \, : \, \det v = 1 \right\}
\end{equation}

The multiplication theorem for weighted Sobolev spaces (cf. Lemma 2.5 of \cite{ChoquetChristo81}) shows that both $\mathcal{G}_\delta (E)$ and $\mathcal{G}_\delta^\C (E)$ have a group structure, and that there are well-defined actions of $\mathcal{G}_\delta (E)$ on $\mathcal{A}_\delta$ and $\mathcal{G}_\delta^\C(E)$ on $\mathcal{A}_\delta^\C (E)$ respectively.

Similarly we have balanced versions of these spaces $\mathcal{G}_{bal, \delta}(E)$, $\mathcal{A}_{bal, \delta}(E)$ and $\Omega_{bal, \delta}^1(\ad(E))$. When a  smooth pair $(d_A, \psi) \in \mathcal{A}_{bal, \delta}(E) \times \Omega_{bal, \delta}^1(\ad(E))$ solves the equations \eqref{eqn:mu1'} and \eqref{eqn:mu2'}, then the \emph{holonomy of the pair} $(d_A, \psi)$ refers to the holonomy of the flat connection $d_A + \psi \in \mathcal{A}_{bal, flat, \delta}^\C(E)$.

\begin{proposition}\label{prop:exists-smooth}
If   $D_i \in \mathcal{A}_{bal, flat, \delta}^\C (E)$, $i=1,2$ are  smooth and $\mathcal{G}_{bal, \delta}^\C(E)$-gauge equivalent then  they are $\mathcal{G}_{bal}^\C(E)$-gauge equivalent. 
\end{proposition}
\begin{proof} Since the result is local, it follows by elliptic regularity.
\end{proof}

\begin{proposition}\label{prop: trivial-hol}
Let $D \in  \mathcal{A}_{bal, flat, \delta}^\C (E)$ be  smooth. Then  $D$ has trivial holonomy around the vertices of $X$.
\end{proposition}

\begin{proof}
 For $D=d+A$ write $A(t, \theta)=B(t, \theta)dt+C(t, \theta)d\theta$. Consider the family of loops  $c_t: [0,2\pi] \rightarrow X$ given by $c_t(\theta)=(t, \theta)$   and consider
 the holonomy equation from Definition~\ref{def:holonomy0} along  $c_t( \theta)$ 
\begin{equation}\label{eqn:holonomy0}
\frac{d s_t(\theta)}{d\theta} + C (t, \theta) s_t(\theta) = 0 \ \ \mbox{with} \ \ s_t(0)=id.
\end{equation}
 Lemma IV.4.1 on p54 of \cite{Hartman64} implies
 \begin{equation}\label{eqn:ode-estimate0}
| s_t(\theta) | \leq | s_t(0) |  \exp \left\{ \int_0^{\theta} | C(t, \theta) | \, d\theta \right\} \leq K \exp \left\{ \int_0^{2\pi} | C(t, \theta) | \, d\theta \right\}
\end{equation}
where $K$ is a dimensional constant.
Since
\[
\int_0^\infty e^{t \delta} \int_0^{2\pi}| C(t, \theta) |^2 d\theta dt < \infty
\]
there exists a sequence $t_i \rightarrow \infty$ such  $\int_0^{2\pi}| C(t_i, \theta) |^2 d\theta \rightarrow 0$. By Cauchy-Schwarz  we also have
\begin{equation}\label{eqn:L2estimate}
\int_0^{2\pi}|C(t_i, \theta) | d\theta \rightarrow 0.
\end{equation}
Combined with (\ref{eqn:ode-estimate0}) this implies  that $| s_{t_i}(\theta) |$ is uniformly bounded.
By integrating equation (\ref{eqn:holonomy0}) with respect to $\theta$, we obtain from (\ref{eqn:L2estimate})
\begin{eqnarray}
|s_{t_i}(2\pi)-s_{t_i}(0)| \leq \int_0^{2\pi}| s_{t_i}(\theta) || C(t_i, \theta) | d\theta \rightarrow 0.
\end{eqnarray}
Since the holonomy is independent of $t$ we obtain that $s_{t_i}(2\pi)=s_{t_i}(0)$ and thus it must be trivial.
\end{proof}

Proposition~\ref{prop:exists-smooth} and Proposition~\ref{prop: trivial-hol} allow us to define the notion of conjugacy class of holonomy for a smooth flat connection $D \in \mathcal{A}_{bal, flat, \delta}^{\C, irr} (E)$ as follows.
\begin{definition}\label{def:holonomy} 
Let $D \in \mathcal{A}_{bal, flat, \delta}^{\C, irr} (E)$  be a  smooth flat connection  and let $\rho_* : \pi_1(X_*) \rightarrow \SL(r, \C)$ be the  holonomy of $D$, where $X_*=X \backslash X^0$ and $X^0 $ denotes the 0-skeleton of $X$. Since the star of a vertex is contractible, then Van Kampen's theorem implies that $\pi_1(X) =\pi_1(X_*) \slash \pi$, where $\pi$ denotes the subgroup of $ \pi_1(X_*)$ generated by
 $\cup_{p \in X_0}  \pi_1(Lk(p))$. By Proposition~\ref{prop: trivial-hol}, the restriction of $\rho_*$ to $\pi$  is trivial hence it induces a homomorphism $\rho : \pi_1(X) \rightarrow \SL(r, \C)$. We say that the conjugacy class of holonomy of $D$ is $[\rho]$. Notice that the map is well defined since  gauge equivalent pairs yield conjugate holonomies. Furthermore, $\rho$ is irreducible because $D$ is irreducible.
\end{definition}

\section{Equivalence of moduli spaces}\label{sec:equivalence}

\subsection{Higgs Moduli Space}\label{subsec:dolbeaultspace}
We fix a vector bundle $E_c=E$ of rank $r$ over a $2$-complex $X$ with a Hermitian metric, and a fixed Riemannian metric on the base space $X$. 
\begin{definition}The \emph{ Higgs moduli space} is the space $\mathcal{M}_{Higgs}(E)$ of $\mathcal{G}_{bal, \delta }(E)$ equivalence classes of pairs $(d_A, \psi) \in \mathcal{A}_{bal, \delta}(E) \times \Omega_{bal, \delta}^1(\sqrt {-1}\ad (E))$ that are \emph{smooth, irreducible} and  solve the following equations
\begin{align}
F_A + \psi \wedge \psi & = 0 \label{eqn:mu1} \\
d_A \psi & = 0 \label{eqn:mu2} \\
d_A^* \psi & = 0 \label{eqn:mu3}.
\end{align}
We endow $\mathcal{M}_{Higgs}(E)$ with the $L^2_{1, \delta}$-topology.
\end{definition}

Given $[(d_A, \psi)] \in \mathcal{M}_{Higgs}(E)$, we can  assign by Definition~\ref{def:holonomy} the holonomy  $[\rho]$ of the flat connection $d_A+\psi$ and set  $\alpha[(d_A, \psi)] := [\rho]$. The map $\alpha$ is well defined. The next proposition follows from continuous dependence of solutions of ODE upon the initial condition.

\begin{proposition}\label{prop:BettiHiggs}
The  map $\alpha : \mathcal{M}_{Higgs}(E) \rightarrow \mathcal{M}_{char}^c $, where $\alpha[(d_A, \psi)]= [\rho]$, is well defined and continuous.
\end{proposition}
The following is the main theorem of this paper.
\begin{theorem}\label{thm:BettiHiggshomeomorhism}
The  map $\alpha : \mathcal{M}_{Higgs}(E) \rightarrow \mathcal{M}_{char}^c $ is a  homeomorphism.
\end{theorem}

In the next section we will construct the inverse map. We end this section with a  proposition that will be used later. 

\begin{proposition}\label{prop:injective}
Let $(d_{A_1}, \psi_1)$ and $(d_{A_2}, \psi_2)$ be solutions to equations (\ref{eqn:mu1})-(\ref{eqn:mu3}) and assume that they are $ \mathcal{G}^{\C}_{bal, \delta}(E)$-gauge equivalent. Then they are $ \mathcal{G}_{bal, \delta}(E)$-gauge equivalent.
 \end{proposition}

\begin{proof}
Assume that  there exists $g \in \mathcal{G}_{bal, \delta}^\C(E)$ such that $(d_{A_1}, \psi_1) = g \cdot (d_{A_2}, \psi_2)$, and we have to show that $g$ is unitary. 
Let $h=g^*g$ and we will show that $h$ is constant. By \cite{Simpson88} Lemma 3.1(d) we have the following point wise estimate away from the vertices (notice that the sign of our Laplacian is the opposite from Simpson's)
\begin{equation}\label{eqn:equality-iff}
 \quad \Delta \tr (h)  \leq 0.
\end{equation}
 Now since $g$ is balanced, then so is $\tr h$, and therefore an application of Stokes' theorem on each face of $X$ shows that
\begin{align}
\begin{split}
\int_X \Delta \tr h \, dx & = \lim_{r \rightarrow 0} \int_{X \setminus \bigcup_{\text{0-cells} \, v} B_r(v)} \Delta \tr h\, dx \\
& = \lim_{r \rightarrow 0} \sum_{\text{2-cells} \, \sigma} \int_{F \setminus \bigcup_{\text{0-cells} \, v} B_r(v)} \Delta \tr h\, dx \\
& = \lim_{r \rightarrow 0} \sum_{\text{2-cells} \, \sigma} \int_{\partial \left( F \setminus \bigcup_{\text{0-cells} \, v} B_r(v) \right)} \frac{\partial \tr h}{\partial \nu} \, ds
\end{split}
\end{align}
where $\nu$ is the outwards pointing normal vector on $\partial \left( \sigma \setminus \bigcup_{\text{0-cells} \, v} B_r(v) \right)$. The boundary $\partial \left( \sigma \setminus \bigcup_{\text{vertices} \, v} B_\sigma(v) \right)$ consists of points on the 1-cells of $\sigma$, and points on $\partial B_r(v) \cap \sigma$. Breaking the integral into these two parts, we obtain
\begin{multline}
\sum_{\text{2-cells} \, \sigma} \int_{\partial \left( \sigma \setminus \bigcup_{\text{0-cells} \, v} B_r(v) \right)} \frac{\partial \tr h}{\partial \nu} \, ds = \sum_{\text{2-cells} \, \sigma} \left( \sum_{\text{1-cells} \, \tau \, : \tau \cap \bar{\sigma} \neq \emptyset} \int_{\tau \setminus \bigcup_v B_r(v) \cap \tau} \frac{\partial \tr h}{\partial \nu} \, ds \right) \\ + \sum_{\text{2-cells} \, \sigma} \int_{\bigcup_v \partial B_r(v) \cap \sigma} \frac{\partial \tr h}{\partial \nu} \, ds
\end{multline}
The balancing condition shows that the first term is zero. Therefore we are left with
\begin{equation}\label{eqn:reduced-laplacian-eqn}
\int_X \Delta \tr h \, dx  = \lim_{r \rightarrow 0}  \sum_{\text{2-cells} \, \sigma} \int_{\bigcup_v \partial B_r(v) \cap \sigma} \frac{\partial \tr h}{\partial \nu} \, ds
\end{equation}
In polar co-ordinates, each component of this integral becomes
\begin{equation}
\int_{\partial B_r(v) \cap F} \frac{\partial \tr h}{\partial \nu} \, ds = r \int_0^{\frac{\pi}{3}} \frac{\partial \tr h}{\partial r} \, d \theta 
\end{equation}
Since $h \in \mathcal{G}(E)_{bal, \delta}^\C$ (and in particular, the integral of $\frac{\partial^2 h}{\partial r^2}$ is bounded), then we have
\begin{equation}
\lim_{r \rightarrow 0} \left( \sigma \int_0^{\frac{\pi}{3}}  \tr \left( \frac{\partial h}{\partial r} \right) \, d \theta  \right) = 0
\end{equation}
and so  \eqref{eqn:reduced-laplacian-eqn} becomes
\begin{equation}
\int_X \Delta \tr h \, dx = 0.
\end{equation}
Combined with $\Delta \tr h \leq 0$ from \eqref{eqn:equality-iff}, we see that $\Delta \tr h = 0$.  The second to the last formula in \cite[p876]{Simpson88} implies that $D(h)=0$ pointwise away from the vertices. This implies 
 that the connection $D$ splits according to the eigenspaces of $h$ , and since the connection $D$ is indecomposable, then $h$ must be a constant multiple of the identity matrix, which concludes the proof.
\end{proof}

\subsection{The inverse map}\label{subsec:the inverse map}

 For an irreducible representation $\rho : \pi_1(X) \rightarrow \SL(r, \C)$, with $[\rho] \in  \mathcal{M}_{char}^c$ and $E=E_c$, Theorem \ref{thm:harmonic-map-existence} then shows that there exists a unique $\rho$-equivariant harmonic map $u : \tilde{X} \rightarrow \SL(r, \C) / \SU(r)$.
As in subsection~\ref{subsec:relate-higgs}, let $d_A$ and $\psi$ be the associated unitary connection and Higgs field. Since $u$ is harmonic, $d_A$ is the pullback of the Levi-Civita connection on $\SL(r, \C) / \SU(r)$, and $\psi$ is the derivative of $u$, then we also have the equation
\begin{equation}
d_A^* \psi = 0 \label{eqn:momentmap3}
\end{equation}
almost everywhere (in fact by Theorem~\ref{thm:harmonic-map-existence} everywhere away from the 0-skeleton). 

\begin{proposition}\label{prop:rightspace} 
If  $u$ is harmonic,  $\alpha$ as in Theorem \ref{thm:harmonic-map-existence} and $\delta < \alpha$, then $\mathcal D \in \mathcal{A}_{bal, flat, \delta}^\C(E)$. The metric on the bundle $E$ induces a decomposition of $\mathcal D$ into skew-adjoint and self-adjoint parts, $\mathcal D = d_A + \psi$, where $d_A \in \mathcal{A}_{bal, \delta}(E)$ and $\psi \in \Omega^1_{bal, \delta}(i \ad(E))$. Furthermore, $\mathcal D$, $d_A $ and $ \psi$ are smooth (over $X_*$).
\end{proposition}

\begin{proof}
The construction in Section \ref{subsec:relate-higgs} shows that the connection $\mathcal D$ is induced from the trivial connection on the universal cover, hence it is clearly balanced, flat  and $L_{1, \delta}^2$. Furthermore, since $d_A = u^*\nabla$ and $\psi=u^{-1}du$, Theorem~\ref{thm:harmonic-map-existence} and (\ref{eqn:fund-eqn}) imply that $d_A$ and $\psi$ are balanced. Therefore, since $u : X \rightarrow \SL(r, \C) / \SU(r)$ is a Lipschitz map over the compact space $X$, in order to show that $d_A \in \mathcal{A}_{bal, \delta}(E)$ and $\psi \in \Omega_{bal}^1(i \ad(E))_\delta$, it suffices to show that  to show  that $du \in L_{1, \delta}^2$.

First we show that $du \in L_\delta^2$. Theorem \ref{thm:harmonic-map-existence} shows that $\left| \frac{\partial u}{\partial r} \right| \leq C r^{\alpha - 1}$ and $\left| \frac{\partial u}{\partial \theta} \right| \leq C r^\alpha$ for some positive $\alpha$. Using the co-ordinate transformation $r = e^{-\tau}$ we see that $\left| \frac{\partial u}{\partial \theta} \right| \leq C e^{-\alpha \tau}$, and
\begin{align*}
\left| \frac{\partial u}{\partial \tau} \right| & = \left| \frac{\partial u}{\partial r} \frac{d r}{d \tau} \right| \\
 & \leq C r^{\alpha - 1} r \\
 & = C e^{-\alpha \tau}
\end{align*}

Therefore, $du \in L_\delta^2$ if $\delta < \alpha$. Similarly, we use the estimates on the second derivatives of $u$ to show that $du \in L_{1, \delta}^2$. We have $\left| \frac{\partial^2 u}{\partial \theta^2} \right| \leq C e^{-\alpha \tau}$, and we can compute
\begin{align*}
\left| \frac{\partial^2 u}{\partial \tau \partial \theta} \right| & = \left| \frac{\partial^2 u}{\partial r \partial \theta} \frac{d r}{d \tau} \right| \\
 & \leq C r^{\alpha - 1} r \\
 & = C e^{-\alpha \tau}
\end{align*}
and similarly
\begin{align*}
\left| \frac{\partial^2 u}{\partial \tau^2} \right| & = \left| \frac{\partial}{\partial \tau} \left( \frac{\partial u}{\partial r} \frac{d r}{d\tau} \right) \right| \\
 & = \left| \frac{\partial^2 u}{\partial r^2} \left( \frac{d r}{d \tau} \right)^2 + \frac{\partial u}{\partial \tau} \left( \frac{d}{dr} \frac{d r}{d \tau} \right) \frac{d r}{d \tau} \right| \\
 & \leq C r^{\alpha - 2} r^2 + C r^{\alpha - 1} r^2 \\
 & \leq C e^{-\alpha \tau} + C e^{-(\alpha+1)\tau} \\
 & \leq C e^{-\alpha \tau},
\end{align*}
where in the last step we use the fact that $\tau \geq 0$ near a vertex. Therefore, $du \in L_{1, \delta}^2$ if $\delta < \alpha$.
\end{proof}

\begin{theorem}\label{thm:BettiHiggs}
The  map $\beta :\mathcal{M}_{char}^c \rightarrow \mathcal{M}_{Higgs}(E)$ defined by $\beta([\rho])=[(d_A, \psi)]$ is a  continuous inverse of $\alpha$.
\end{theorem}

\begin{proof}
The first step is to show  that the map $\beta$ is well defined. Given $\rho$, \ Proposition~\ref{prop:rightspace} implies that $d_A \in \mathcal{A}_{bal, \delta}(E)$ and $\psi \in \Omega_{bal}^1(i \ad(E))_\delta$. Moreover, we claim that the pair $(d_A, \psi)$ is irreducible. If  $\rho_*: \pi_1(X_*) \rightarrow \SL(r,\C)$ denotes the holonomy of the flat connection $d_A + \psi$ then,  as pointed out in Definition~\ref{def:holonomy},  $\rho_*=\rho \circ p $, where $p: \pi_1(X_*) \rightarrow \pi_1(X) =\pi_1(X_*) \slash \pi$ is the natural quotient map. Since by assumption  $\rho$ is irreducible it follows that $\rho_*$ is also irreducible proving our claim.

Now, let $\rho$ and $\rho'=\gamma \rho \gamma^{-1}$ be two representatives of $[\rho]$ and  let $u$ and $u'$ be the two corresponding equivariant harmonic maps. It follows that $u'=\gamma \cdot u$ where $\cdot$ denotes the action of $\SL(r, \C)$ on $\SL(r, \C) / \SU(r) $. It follows that the induced decompositions $\mathcal D= d_A + \psi$ on the universal cover
 agree, hence after taking the quotients by $\rho$ and $\rho'=\gamma \rho \gamma^{-1}$ respectively the corresponding pairs are complex gauge equivalent by $ \gamma$. Proposition~\ref{prop:injective} then shows that they are $\mathcal{G}_{bal, \delta}$ gauge equivalent which completes the proof that 
$\beta$ is well defined.

Next we will  first show that $\alpha(\beta([\rho]))=[\rho]$. Let $\beta([\rho]))=[(d_A, \psi)]$. According to \eqref{eqn:fund-eqn}, we have  $d_A+\psi=\mathcal D$, where $\mathcal D$ is the connection on $ad(E_\rho)$ induced by the trivial connection on the universal cover which has holonomy $\rho$. Hence,  $\alpha(\beta([\rho]))=[\rho]$.

Conversely, $\beta(\alpha([(d_A, \psi)]))=[(d_A, \psi)]$. Indeed, let $(d_B, \phi)$ be a smooth representative of $\beta(\alpha([(d_A, \psi)]))$. By applying $\alpha$ on both sides and what we just proved,
$ \alpha([(d_A, \psi)])=\alpha([(d_B, \phi)])$. In other words, $(d_A, \psi)$ and $(d_B, \phi)$ have conjugate holonomies. Since the holonomies of these pairs near the vertices are trivial by  Proposition~\ref{prop: trivial-hol},  then Proposition~\ref{prop:holonomyequivalence} implies that the corresponding flat connections (and hence also the pairs) are complex gauge equivalent. Thus Proposition~\ref{prop:injective} implies that $(d_A, \psi)$ and $(d_B, \phi)$ are $\mathcal{G}_{bal, \delta}$ gauge equivalent, hence $\beta(\alpha([(d_A, \psi)]))=[(d_A, \psi)]$.

In order to prove continuity, let $\rho_i \rightarrow \rho \in \mathcal{M}_{char}^c$ and let $u_i$, $u$ the associated equivarient harmonic maps. Fix a compact fundamental domain $F \subset \tilde X$  for the action of $\Gamma$ and define $\rho_i $ equivarient maps 
${\tilde u}_i$ by setting ${\tilde u}_i=u$ on $F$ and extending $\rho_i $ equivalently on $\tilde X$. Since $u_i$ are harmonic, we have
\[
E^{u_i} \leq E^{\tilde{u}_i}=E^u.
\]
The global H\"older bound (cf. Theorem 3.12 of \cite{DaskalMese06}) implies that there is a subsequence (we call it again by $\{i \}$ by a slight abuse of notation) such that $u_i \rightarrow u_\infty$  uniformly on $F$. Furthermore, the convergence of the representations $\rho_i \rightarrow \rho $ implies that $u_\infty$ is $\rho$-equivariant and Theorem 5.1 of \cite{DaskalMese06} implies that $u_\infty$ is harmonic. Finally, the  uniqueness Theorem 4.6 of \cite{DaskalMese06} implies that $u_\infty=u$. We have thus shown so far
\[
u_i \rightarrow u \ \ \mbox{locally uniformly}.
\]
Let $(d_{A_i}, \psi_i)$ denote the unitary connection and Higgs field associated with the harmonic map $u_i$.  By Theorem~\ref{thm:harmonic-map-existence} together with the proof of Proposition~\ref{prop:rightspace} (in this we use the third derivative estimates) we obtain that the  $L^2_{2, \delta}$-norm of $({A_i}, \psi_i)$ is uniformly bounded, and thus there exists a subsequence (we call it again by $\{i \}$ by a slight abuse of notation) such that $(d_{A_i}, \psi_i)
\rightarrow (d_A, \psi)$ weakly in $L^2_{2, \delta}$ and hence strongly in $L^2_{1, \delta}$.
\end{proof}

\def\cprime{$'$}

\end{document}